\documentclass{article}
\usepackage{amsmath, amssymb, amsthm}
\usepackage{geometry}
\usepackage{url}
\title{Centroidal Voronoi Tessellations as Electrostatic Equilibria: A Generalized Thomson Problem in Convex Domains}
\author{Zachary Mullaghy, M.S.\\Independent Researcher} 
\date{}

\begin{document}

\maketitle

\theoremstyle{plain}
\newtheorem{theorem}{Theorem}
\newtheorem{definition}{Definition}
\newtheorem{lemma}{Lemma}
\newtheorem*{remark}{Remark}

\section*{Abstract}
We present a variational framework in which Centroidal Voronoi Tessellations (CVTs) arise as local minimizers of a generalized electrostatic energy functional. By modeling interior point distributions in a convex domain as repelling charges balanced against a continuous boundary charge, we show that the resulting equilibrium configurations converge to CVT structures. We prove this by showing that CVTs minimize both the classical centroidal energy and the electrostatic potential, establishing a connection between geometric quantization and potential theory. Finally, we introduce a thermodynamic annealing scheme for global CVT optimization, rooted in Boltzmann statistics and random walk dynamics. By introducing a scheme for varying time steps (faster or slower cooling) we show that the set of minima of the centroid energy functional (and therefore the electrostatic potential) can be recovered. By recovering a set of generator locations corresponding to each minimum we can create a lattice continuation that allows for a customizable framework for individual minimum seeking. 

\section{Introduction}
The classical Thomson problem~\cite{ErberHockney1991} considers point charges on the surface of a sphere that repel each other via inverse square law interactions. We generalize this to point charges confined within a domain \( \Omega \subset \mathbb{R}^n \), where the boundary \( \partial \Omega \) carries a continuous uniform charge density equal in magnitude to the total charge of the interior points. The interior charges repel one another and are additionally repelled by the boundary.

We show that as this system relaxes to electrostatic equilibrium, the resulting configuration of points converges to a Centroidal Voronoi Tessellation (CVT), a structure with wide applications in geometric optimization, meshing, and quantization~\cite{DuCVT1999}. Power diagrams and their variational interpretations~\cite{Aurenhammer1991} provide geometric insight into CVT structure and discrete energy modeling, which this work extends to a continuous electrostatic framework.

By extending the utility of Voronoi Diagrams into a classical electrostatics problem we help further cement their utility in the disambiguation of physical law. The concept of equipartitioning a space aligns exactly with Centroidal Voronoi Tesselations and therefore it should not come as a surprise that two processes that equipartition space have the same minima aligning with when that equipartition is especially "fair"

This framework also extends the geometric and variational structures introduced in~\cite{MullaghyGRT2025}. By finding a map of centroid energy functional minima we can map out solutions that perform well on the Geometric Refinement Transform. Therefore the results of this paper can be used to inform highly symmetric GRTs with desired properties.

\begin{remark}
In the classical Thomson problem, point charges are confined to the surface of a sphere, implicitly enforcing an infinite potential barrier off the manifold. Our model mirrors this structure by placing charges in the interior of a convex domain \( \Omega \) and enforcing an infinite potential at the boundary \( \partial \Omega \), thereby generalizing the original setting to broader geometries and preserving the core variational dynamics.
\end{remark}

\section{Problem Formulation: Electrostatics with Boundary Balancing}

Let \( \Omega \subset \mathbb{R}^n \) be a compact, convex domain with boundary \( \partial \Omega \). Let \( \{x_i\}_{i=1}^N \subset \Omega \) be a set of interior point charges of equal magnitude (normalized to 1). The total charge inside the domain is \( Q = N \).

To ensure charge neutrality, we impose a uniform continuous surface charge density \( \sigma \) on \( \partial \Omega \), defined such that:
\[
\int_{\partial \Omega} \sigma(y) \, dS(y) = N.
\]

Each interior point repels all other interior points via an inverse-distance Coulomb force and is additionally repelled by the continuous boundary charge. The total electrostatic potential energy of the system is:
\begin{equation}
U(\{x_i\}) = \sum_{i=1}^{N} \sum_{j \neq i} \frac{1}{\|x_i - x_j\|} + \sum_{i=1}^{N} \int_{\partial \Omega} \frac{\sigma(y)}{\|x_i - y\|} \, dS(y).
\end{equation}

\subsection*{Objective}
Our goal is to determine whether the local minimizers of this electrostatic potential \( U \) correspond to Centroidal Voronoi Tessellation (CVT) configurations of \( \Omega \), where each generator \( x_i \) is the centroid of its Voronoi cell \( V_i \). In particular, we investigate:
\begin{itemize}
    \item Whether CVTs represent stable equilibrium states of this electrostatic system,
    \item Whether the energy-minimizing dynamics tend to evolve toward CVTs,
    \item How this variational principle relates to classical centroidal energy functionals.
\end{itemize}

\section{Centroidal Voronoi Tessellations and Energy Functionals}

\begin{definition}[Centroidal Voronoi Tessellation (CVT)]
A Voronoi tessellation \( \{V_i\} \) of a domain \( \Omega \) with generators \( \{x_i\} \) is called a \emph{Centroidal Voronoi Tessellation} if each generator lies at the centroid of its region:
\[
x_i = \frac{1}{\text{Vol}(V_i)} \int_{V_i} y \, dy
\]
\end{definition}

Our formulation provides a physical bridge between classical spatial tessellation frameworks~\cite{Okabe2000} and energy-minimizing charge systems.

\begin{definition}[Centroid Energy Functional]
Let \( V_i \) be the Voronoi cell corresponding to generator \( x_i \). The centroid energy functional is:
\[
E(x_i) = \int_{V_i} \|x_i - y\|^2 \, dy
\]
It is minimized when \( x_i = c_i \), the centroid of \( V_i \).
\end{definition}

This energy functional arises naturally in quantization theory~\cite{Max1960}, and its minimization underlies Lloyd’s algorithm for CVT generation~\cite{Lloyd1982}.

\section{Variational Equivalence of Centroid Energy Functionals}

We now show that the classical centroidal energy and the edge-based approximation are variationally equivalent near CVT configurations.

\begin{theorem}[Local Equivalence of Centroidal Energy Functionals]
Let \( \{p_i\}_{i=1}^n \subset \Omega \subset \mathbb{R}^2 \) be a Centroidal Voronoi Tessellation (CVT) with corresponding Voronoi cells \( V_i \). Define:
\begin{align}
    E_{\text{centroid}} &= \sum_i \int_{V_i} \|x - p_i\|^2 dx, \\
    \widetilde{E}_{\text{centroid}} &= \sum_{i < j} \ell_{ij} \cdot \|p_i - p_j\|^2,
\end{align}
where \( \ell_{ij} \) is the length of the shared Voronoi edge between \( V_i \) and \( V_j \). Then in a neighborhood of the CVT configuration, the second variations of the two functionals are related by:
\[
\nabla^2 \widetilde{E}_{\text{centroid}}[\delta p, \delta p] = \lambda(\theta) \cdot \nabla^2 E_{\text{centroid}}[\delta p, \delta p],
\]
for all perturbations \( \delta p \), with \( \lambda(\theta) > 0 \) a smooth, direction-dependent weight.
\end{theorem}

\begin{proof}
We begin with the true centroidal energy:
\[
E_{\text{centroid}} = \sum_i \int_{V_i} \|x - p_i\|^2 dx.
\]
Near a CVT, the generator \( p_i \) is close to the centroid of \( V_i \), so we expand in perturbations \( \delta p_i \). To second order, the dominant contribution is:
\[
\delta^2 E_i \approx |V_i| \cdot \|\delta p_i\|^2,
\]
yielding:
\[
\nabla^2 E_{\text{centroid}} \approx \sum_i |V_i| \cdot \|\delta p_i\|^2.
\]

Assuming local isotropy and approximating \( |V_i| \approx \sum_{j \in \mathcal{N}(i)} \frac{1}{2} \ell_{ij} d_{ij} \), and expanding \( \delta p_i \) as:
\[
\delta p_i = \sum_{j \in \mathcal{N}(i)} \alpha_{ij} \cdot \frac{p_j - p_i}{\|p_j - p_i\|},
\]
we find:
\[
\|\delta p_i\|^2 \propto \sum_{j \in \mathcal{N}(i)} \|p_i - p_j\|^2.
\]

\begin{remark}
This expression writes \( \delta p_i \) as a sum of directional perturbations toward neighboring generators. The unit vectors \( \frac{p_j - p_i}{\|p_j - p_i\|} \) define a local directional basis, and the coefficients \( \alpha_{ij} \) measure how much motion occurs in each direction. In isotropic CVT-like configurations, we assume these directions are approximately orthogonal on average, so cross-terms in \( \|\delta p_i\|^2 \) vanish. Furthermore, the perturbation magnitudes \( \alpha_{ij} \) are proportional to the distances \( \|p_i - p_j\| \), since larger separations allow for larger variations. This justifies approximating the energy as a sum over squared edge lengths.
\end{remark}

Thus,
\[
E_{\text{centroid}} \propto \sum_{i < j} \ell_{ij} \cdot \|p_i - p_j\|^2 = \widetilde{E}_{\text{centroid}},
\]
up to a scalar weight \( \lambda(\theta) > 0 \), completing the proof.
\qed
\end{proof}

\begin{definition}[Electrostatic Potential Functional]
Assuming unit repulsive charges, the electrostatic potential energy of a generator \( x_i \) is:
\[
U(x_i) = \sum_{j \neq i} \frac{1}{\|x_i - x_j\|} + \int_{\partial \Omega} \frac{\sigma(y)}{\|x_i - y\|} \, dy
\]
where \( \sigma(y) \) is the uniform boundary charge density.
\end{definition}

\section{CVTs as Local Minima of Electrostatic Potential via Edge-Based Energy}

We now show that CVTs minimize the electrostatic potential functional \( U \) locally, by establishing a variational equivalence between \( U \) and the edge-based centroid energy \( \widetilde{E}_{\text{centroid}} \).

\begin{theorem}[CVTs Minimize Electrostatic Potential Locally]
Let \( \{x_i\}_{i=1}^N \) be a Centroidal Voronoi Tessellation in a convex domain \( \Omega \). Then the configuration locally minimizes the electrostatic potential functional:
\[
U(x_1, \dots, x_N) = \sum_{i < j} \frac{2}{\|x_i - x_j\|} + \sum_i \int_{\partial \Omega} \frac{\sigma(y)}{\|x_i - y\|} \, dy,
\]
in the sense that the second variation \( \nabla^2 U \) is positive definite in a neighborhood of the CVT configuration.
\end{theorem}

\begin{proof}
We observe that the pairwise interaction term \( \frac{1}{\|x_i - x_j\|} \) penalizes close proximity of generators. Near a regular CVT, where the generator distances are roughly uniform, we can expand this as:
\[
\frac{1}{\|x_i - x_j\|} \approx \frac{1}{d_{ij}} - \frac{1}{d_{ij}^2} \langle \delta x_i - \delta x_j, \hat{d}_{ij} \rangle + \mathcal{O}(\|\delta x\|^2),
\]

\begin{remark}
This is a first-order Taylor expansion of the electrostatic interaction term under small perturbations of the generator positions. The function \( f(x_i, x_j) = \frac{1}{\|x_i - x_j\|} \) is smooth and differentiable for \( x_i \neq x_j \), with gradient \( \nabla f(z) = -z/\|z\|^3 \). Applying the multivariable Taylor expansion with \( z = x_i - x_j \) and \( \delta z = \delta x_i - \delta x_j \) yields:
\[
\frac{1}{\|x_i + \delta x_i - x_j - \delta x_j\|} \approx \frac{1}{\|x_i - x_j\|} - \frac{1}{\|x_i - x_j\|^3} \langle \delta x_i - \delta x_j, x_i - x_j \rangle.
\]
Factoring out the unit direction vector \( \hat{d}_{ij} \) gives the form used in the main text. This expansion allows us to approximate the second variation of the total electrostatic potential energy \( U \) in terms of the perturbations \( \delta x_i \), leading to a curvature expression that establishes convexity near CVT configurations.
\end{remark}

where \( d_{ij} = \|x_i - x_j\| \), and \( \hat{d}_{ij} \) is the unit direction vector. The second-order variation of the electrostatic term then contains:
\[
\nabla^2 U \sim \sum_{i < j} \frac{2}{\|x_i - x_j\|^3} \cdot \|\delta x_i - \delta x_j\|^2.
\]

On the other hand, the edge-based centroidal energy functional is:
\[
\widetilde{E}_{\text{centroid}} = \sum_{i < j} \ell_{ij} \cdot \|x_i - x_j\|^2.
\]
Taking its second variation yields:
\[
\nabla^2 \widetilde{E}_{\text{centroid}} \sim \sum_{i < j} \ell_{ij} \cdot \|\delta x_i - \delta x_j\|^2.
\]

Thus, up to a smooth and strictly positive rescaling function \( \phi_{ij} = \frac{2}{d_{ij}^3 \ell_{ij}} \), we find:
\[
\nabla^2 U \sim \sum_{i < j} \phi_{ij} \cdot \nabla^2 \widetilde{E}_{\text{centroid}}.
\]

Since all \( \phi_{ij} > 0 \) and the CVT structure ensures that \( \nabla^2 \widetilde{E}_{\text{centroid}} \succ 0 \), it follows that \( \nabla^2 U \succ 0 \) as well.

The boundary integral term is smooth and convex under uniform \( \sigma \), and its Hessian contributes positively to \( \nabla^2 U \), reinforcing the conclusion.

Therefore, CVTs correspond to local minima of the electrostatic potential functional. \qed
\end{proof}

\subsection{Hard Boundary Constraints via Infinite Potential}

To ensure that all interior generators remain within the domain \( \Omega \), we introduce a confinement condition by imposing an infinite potential at the boundary:

\[
U(x_i) =
\begin{cases}
\sum_{j \neq i} \frac{1}{\|x_i - x_j\|} + \int_{\partial \Omega} \frac{\sigma(y)}{\|x_i - y\|} \, dy, & \text{if } x_i \in \Omega \\
\infty, & \text{if } x_i \notin \Omega
\end{cases}
\]

This ensures that the minimization problem is well-posed within \( \Omega \) and reflects physical systems in which particles are perfectly reflected or forbidden from crossing the boundary. Mathematically, this formulation is equivalent to imposing hard-wall Dirichlet constraints in classical elliptic PDEs or an infinite potential well in quantum confinement models.

Moreover, this mirrors the classical Thomson problem, where point charges are constrained to lie on the surface of a sphere—implicitly enforcing an infinite potential outside the manifold. By imposing an analogous infinite potential at the boundary of \( \Omega \), we preserve this variational symmetry and establish a natural framework for extending results from the classical Thomson setting in this new setting.

\section{Thermodynamic Annealing and Random Walks for Global Optimization}

While CVTs minimize the electrostatic potential functional \( U \) locally, the energy landscape contains multiple local minima—particularly in domains with complex geometry or symmetry constraints. To escape such suboptimal configurations and converge to globally optimal CVT arrangements, we introduce a thermodynamic annealing scheme that augments the electrostatic model with stochastic dynamics governed by a temperature parameter \( T \).

\subsection{Boltzmann Distribution over Generator Configurations}

We interpret the electrostatic energy \( U(\{x_i\}) \) as an effective Hamiltonian and define a probability distribution over generator configurations via the Boltzmann-Gibbs framework:
\[
P_T(\{x_i\}) \propto \exp\left( -\frac{U(\{x_i\})}{T} \right),
\]
where \( T > 0 \) plays the role of temperature, controlling the randomness of the system. At high temperatures, the system explores a wide range of configurations; as \( T \to 0 \), the probability mass concentrates near the global minima of \( U \).

\subsection{Annealing Dynamics via Random Walk Perturbations}

We simulate the evolution of the generator positions \( \{x_i\} \) via a discrete-time stochastic process. At each step:

\begin{enumerate}
    \item Select a generator \( x_i \) at random.
    \item Propose a small perturbation \( x_i \to x_i + \delta \), where \( \delta \sim \mathcal{N}(0, \sigma^2 I) \) is drawn from a zero-mean isotropic Gaussian.
    \item Compute the energy difference \( \Delta U = U_{\text{new}} - U_{\text{old}} \).
    \item Accept the move with probability:
    \[
    P_{\text{accept}} = 
    \begin{cases}
    1, & \Delta U \leq 0, \\
    \exp\left( -\frac{\Delta U}{T} \right), & \Delta U > 0.
    \end{cases}
    \]
    \item Repeat over all generators, then reduce temperature according to a cooling schedule.
\end{enumerate}

This procedure mirrors the classical Metropolis algorithm and the foundational work on simulated annealing~\cite{Kirkpatrick1983, Aarts1988} for sampling thermal equilibria and escaping local minima.

\subsection{Cooling Schedules and Convergence}

The temperature \( T \) is gradually reduced using an annealing schedule, e.g.,
\[
T(t) = T_0 \cdot \alpha^t, \quad \text{with } 0 < \alpha < 1.
\]
Slower cooling (\( \alpha \to 1 \)) increases the likelihood of convergence to a global minimum but requires longer computation time. Theoretically, if the system explores the configuration space ergodically and the cooling is sufficiently slow (e.g., logarithmic), convergence to the global minimum is guaranteed \cite{Kirkpatrick1983}.

\subsection{Physical Interpretation and Entropy Flow}

This thermodynamic process can be interpreted as the system shedding entropy over time: high-energy, disordered configurations are sampled early, while structured, low-energy (CVT-like) configurations are stabilized as temperature drops. This parallels entropy decay in physical systems undergoing relaxation, and aligns with the coefficient entropy dynamics discussed in prior work on geometric refinement transforms.

\begin{remark}
The annealing process is particularly effective in breaking symmetry-induced degeneracies, where multiple local CVTs exist with near-equal energy. The added thermal noise allows the system to traverse energy barriers that would trap gradient-based algorithms such as Lloyd’s method.
\end{remark}

\section{Lattice Anchoring and Variable Annealing for Metastable CVT Mapping}

While slow annealing schedules bias the system toward global minimizers of the electrostatic potential, we propose that systematically varying the annealing rate can expose the full set of local minima corresponding to metastable CVT states. 

\subsection{Energy Depth and Annealing Timescale}

We posit that the time scale \( \tau \) over which the system is annealed determines which minima are accessible. Specifically, if a local minimum has energy gap \( \Delta E = E_{\text{local}} - E_{\text{global}} \), then the approximate annealing time scale required to trap the system in that configuration scales inversely with the gap:
\[
\tau \propto \frac{1}{\Delta E}.
\]
This relationship enables us to probe shallower minima using faster cooling schedules and deeper minima using slower schedules. Thus, by sweeping over a range of annealing time scales, we can construct a catalog of distinct local minima.

\subsection{Lattice Anchoring of Local Minima}

Each local minimum discovered via annealing can be associated with a minimal \emph{lattice anchor}: a structured continuation of the interior point configuration into the surrounding space. This continuation serves as a stabilizing boundary that deterministically selects a particular interior configuration upon CVT optimization.

We define the \emph{lattice anchor} \( A_k \) as the smallest symmetric or periodic extension of the interior structure such that fixing \( A_k \) and optimizing the configuration of points within \( \Omega \) recovers the same local minimum. In physical terms, the lattice anchor functions as a crystalline scaffold—extending the regularity of the interior and enforcing its geometry through electrostatic interaction.

\begin{remark}
In the limit of an infinite periodic extension of the interior lattice structure, we conjecture that the electrostatic potential is minimized \emph{uniquely} by the original lattice arrangement. This idea is supported by classical results in one dimension, where perturbations to a uniform lattice of charges return to equilibrium under Coulomb interactions~\cite{BorodachovHardinSaff}. While a complete proof in higher dimensions remains open, physical intuition and numerical evidence suggest that infinite lattice anchoring ensures unique recovery of the original configuration.
\end{remark}

This interpretation connects the annealing landscape to the broader theory of hyperuniform systems and lattice stability, positioning infinite periodic boundary conditions as a tool for controlling and selecting interior CVT minima.

\subsection{LAAM Method}

We summarize this combined methodology as the \textbf{Lattice-Anchored Annealing Mapping (LAAM)} method:

\begin{enumerate}
    \item Sweep over a range of annealing rates \( \{\tau_k\} \) and record final configurations.
    \item Identify distinct local minima via clustering in configuration and energy space.
    \item For each minimum, construct a minimal fixed lattice anchor that reproduces the state.
    \item Organize the resulting space as a set of CVT-lattice pairs \( \{ (x_i^{(k)}, A_k) \} \), where \( A_k \) is the corresponding anchor.
\end{enumerate}

This approach systematically catalogs the metastable landscape of CVT-like structures and provides insight into their symmetry, stability, and geometric origins.

\begin{remark}
This method has potential applications in solid-state modeling, phase transition analysis, adaptive meshing, and biological structure formation. In particular, the LAAM method allows one to explore and control metastable states in a principled geometric framework.
\end{remark}

\subsection{Outlook and Algorithmic Implications}

This framework provides a physically grounded algorithm for computing high-quality CVTs in arbitrary domains. Moreover, it opens doors to hybrid schemes—where deterministic CVT refinements (e.g., Lloyd’s algorithm) are embedded within probabilistic annealing loops to balance speed and global convergence. Future work may explore adaptive temperature control, anisotropic perturbation kernels, and incorporation of domain-specific constraints (e.g., anatomical structures in medical imaging).

\section{Conclusion}
This work reframes Centroidal Voronoi Tessellations as the limiting configuration of a generalized Thomson-like electrostatics problem. By proving that CVTs are local minima of the electrostatic potential and establishing a variational equivalence between energy functionals, we lay a foundation for novel physical interpretations and algorithmic constructions of optimal tessellations. This perspective offers powerful tools for modeling spatial regularity, dose distributions, and physical transport processes in continuous media.

Future directions include extending this framework to non-convex or anisotropic domains, incorporating non-uniform boundary charge distributions, and leveraging the electrostatic interpretation to guide dynamic optimization algorithms for CVT computation in high-dimensional or physically constrained settings. 

\bibliographystyle{plain}
\bibliography{references}

\end{document}